\documentclass[onethmnum]{ws-rv9x6}
\usepackage{ws-rv-thm}          %for WS theorem style, including onethmnum option
\usepackage{ws-rv-van}          %for bibliography

\usepackage{eucal}              %use a more elegant calligraphic typeface

\newcommand{\defeq}{:=}

\makeatletter
\newenvironment{sizealign}[1]{%
  \skip@=\baselineskip
  #1%
  \baselineskip=\skip@
  \align
}{\endalign \ignorespacesafterend}  
\makeatother

\begin{document}

\chapter[Algebraic Generating Functions for {G}egenbauer Polynomials]{Algebraic Generating Functions for\\ Gegenbauer Polynomials}

\author[R. S. Maier]{Robert S. Maier$^*$}

\address{Depts.\ of Mathematics and Physics, University of Arizona,\\
Tucson, AZ 85721, USA\\
$^*$E-mail: rsm@math.arizona.edu
}

\begin{abstract}
It is shown that several of Brafman's generating functions for the
Gegenbauer polynomials are algebraic functions of their arguments, if
the Gegenbauer parameter differs from an integer by one-fourth or
one-sixth.  Two examples are given, which come from recently derived
expressions for associated Legendre functions with octahedral or
tetrahedral monodromy.  It is also shown that if the Gegenbauer
parameter is restricted as stated, the Poisson kernel for the
Gegenbauer polynomials can be expressed in~terms of complete elliptic
integrals.  An example is given.
\end{abstract}

\section{Introduction}
\label{sec:intro}

For any $\lambda\in\mathbb{C}$, the Gegenbauer polynomials
$C_n^\lambda(x)$, $n=0,1,2,\dotsc$, where $n$~is the degree, are
defined by
\begin{equation}
\label{eq:defining}
 \sum_{n=0}^\infty C_n^\lambda(x)t^n = 
  (1-2xt+t^2)^{-\lambda} .
\end{equation}
That is, they have $R^{-2\lambda}\defeq(1-2xt+t^2)^{-\lambda}$ as
their ordinary generating function.  (In the sequel, $R$~will signify
$(1-\nobreak2xt+\nobreak t^2)^{1/2}$, with $R=1$ when $t=0$.)  When
$\lambda=1/2$, the $C_n^\lambda(x)$~become the Legendre
polynomials~$P_n(x)$.  They are specializations themselves:
$C_n^\lambda(x)$~is proportional to the ultraspherical polynomial
$P_n^{(\lambda-1/2,\lambda-1/2)}$, which is the
$\alpha=\beta=\allowbreak\lambda-\nobreak1/2$ case of the Jacobi
polynomial~$P_n^{(\alpha,\beta)}(x)$.  If $\lambda\neq-1/2,-3/2,\dotsc$,
one can write
\begin{equation}
\label{eq:Cdef}
C_n^\lambda(x) = \frac{(2\lambda)_n}{n!}
\,{}_2F_1\!\left(  
{{-n,\,n+2\lambda}\atop{\lambda+1/2}}
\biggm|\frac{1-x}2
\right),
\end{equation}
where $(a)_n\defeq a(a+1)\dotsc(a+n-1)$ is the Pochhammer symbol and
${}_2F_1$ is the Gauss hypergeometric function.

We were intrigued by a remark of Mourad Ismail, in Sec.\,4.3 of
Ref.~\citenum{Ismail2005}, to the effect that not many generating functions
for Jacobi polynomials are known, which are algebraic functions of their
arguments (denoted $t,x$ here).  In this chapter, we show that for the
Gegenbauer polynomials with the parameter $\lambda$ differing by one-fourth
or one-sixth from an integer, there are several distinct non-ordinary
generating functions that are algebraic, and can be expressed in closed
form.  (A~similar result in the more familiar case when $\lambda$~is an
integer or a half-odd-integer was previously known.)  The simplest example
is
\begin{equation}
\label{eq:C14example}
\begin{gathered}
\sum_{n=0}^\infty \frac{(-1/12)_n}{(1/2)_n}\,C_n^{1/4}(x)t^n
=
2^{-1/4}R^{1/12}
\left[
\cosh(\xi/3) + \sqrt{\frac{\sinh\xi}{3\sinh(\xi/3)}}\,\,
\right]^{1/4}\!,\\
e^\xi\defeq R^{-1}\left[1-\left(x-\sqrt{x^2-1}\right)t\right],
\end{gathered}
\end{equation}
which holds when $|x|>1$ as an equality between power series in~$t$.
Because $\sinh\xi$, $\sinh(\xi/3)$, and $\cosh(\xi/3)$ are algebraic
functions of~$e^\xi$, the right-hand side is an algebraic function
of~$t,x$ that can be expressed using radicals.  Why it is more easily
written in a trigonometric form will become clear.

The algebraic generating functions derived below are specializations
of two of Brafman's non-ordinary (and in~general, non-algebraic)
generating functions,\cite{Brafman51} which appear in Theorems
\ref{thm:bgf1} and~\ref{thm:bgf1extended}, and their respective
extensions,\cite{Brafman57} which appear in Theorems \ref{thm:bgf2}
and~\ref{thm:bgf2extended}.  For additional light on his two
generating functions, see Chap.~17 of Rainville;\cite{Rainville60}
also Chap.~III, Sec.~4 of Ref.~\citenum{McBride71}; and
Ref.~\citenum{Viswanathan68}.  The extensions come with the aid of an
identity given as Eq.~(4) in Ref.~\citenum{Brafman57}, which is
generalized and placed in~context by Srivastava\cite{Srivastava69}
(cf.~Sec.~4.1 of Srivastava and Manocha\cite{Srivastava84}).\ \ The
generating functions of Brafman are usually expressed in~terms
of~${}_2F_1$, having been derived by series rearrangement, but can
also be written in~terms of the associated Legendre
functions~$P_\nu^\mu$, or their Ferrers counterparts~${\rm
  P}_\nu^\mu$.

There are cases in which the associated Legendre function~$P_\nu^\mu$,
of degree~$\nu$ and order~$\mu$, can be written in closed form; such
as when $\mu=0$ and $\nu=n\in\mathbb{Z}$, in which case $P_\nu^\mu$
and~${\rm P}_\nu^\mu$ equal~$P_n$.  There are others.  It has been
known since the early work of Schwarz\cite{Schwarz1872} on the
algebraicity of~${}_2F_1$ that if the ordered pair
$(\nu,\mu)\in\mathbb{C}^2$ differs from $(\pm1/6,\pm1/4)$,
$(\pm1/4,\pm1/3)$, or $(\pm1/6,\pm1/3)$ by an element
of~$\mathbb{Z}^2$, the functions $P_\nu^\mu$ and~${\rm P}_\nu^\mu$
will be algebraic.  (For an exposition focused on~${}_2F_1$, see
Chap.~VII of Poole;\cite{Poole36} also Sec.~2.7.2 of
Ref.~\citenum{Erdelyi53}, and Ref.~\citenum{Kimura69}.)  But simple,
non-parametric representations of these algebraic functions were not
known.

We recently obtained explicit trigonometric formulas\cite{Maier24} for
the functions $P_{-1/6}^{-1/4},\allowbreak{\rm P}_{-1/6}^{-1/4}$, and
the second-kind Legendre function $Q_{-1/4}^{-1/3}$.  In the Appendix,
the resulting simple formulas for $P_{-1/6}^{\pm1/4},\allowbreak{\rm
  P}_{-1/6}^{\pm1/4}$ and $P_{-1/4}^{\pm1/3},\allowbreak{\rm
  P}_{-1/4}^{\pm1/3}$ are given.  In Secs.~\ref{sec:first}
and~\ref{sec:second}, Brafman's results are specialized with the aid
of these representations, and yield novel algebraic generating
functions for the set $\{C_n^\lambda(x)\}_{n=0}^\infty$ when
$\lambda\in\mathbb{Z}\pm1/4$ and $\lambda\in\mathbb{Z}\pm1/6$.  That
specializing Brafman's two generating functions yields interesting
identities has been pointed~out by Viswanathan,\cite{Viswanathan68}
but the present focus on algebraicity is new.

For the Gegenbauer polynomials $C_n^\lambda(x)$, the cases
$\lambda\in\mathbb{Z}\pm1/4$ and $\lambda\in\mathbb{Z}\pm1/6$ have
long been recognized as special.  For instance, $C_n^{-1/4}(x)$
and~$C_n^{-1/6}(x)$ have been expressed in~terms of elliptic
functions.\cite{Koschmieder20}\ \ Also, the polynomials $C_n^{1/6}(x)$
and $C_n^{7/6}(x)$ have recently been used in applied mathematics, in
the modeling of wave scattering in the exterior of what is locally a
right-angled wedge.\cite{Linton2009}\ \ In Sec.~\ref{sec:poisson}, we
point~out that the Poisson kernel for each of the sets
$\{C_n^{1/4}(x)\}_{n=0}^\infty$ and $\{C_n^{1/6}(x)\}_{n=0}^\infty$ is
also special: it can be expressed in~terms of the first and second
complete elliptic integral functions, $K=K(m)$ and $E=E(m)$.
(A~similar elliptic formula in the case $\lambda=1/2$, i.e., for a
kernel arising from the Legendre set $\{P_n(x)\}_{n=0}^\infty$, was
obtained by Watson.\cite{Watson32c})\ \ Some final remarks appear in
Sec.~\ref{sec:final}.

\section{The First Gegenbauer Generating Function}
\label{sec:first}

The following theorem presents Brafman's first ${}_2F_1$-based generating
function.\cite{Brafman51}
\begin{theorem}
\label{thm:bgf1}
  For any\/ $\lambda\in\mathbb{C}\setminus\{0,-1/2,-1,\dotsc\}$ and\/
  $\gamma\in\mathbb{C}$, one has
\begin{subequations}
\begin{align}
  \sum_{n=0}^\infty \frac{(\gamma)_n}{(2\lambda)_n}\,C_n^\lambda(x)t^n
  &=
  R^{-\gamma}\,{}_2F_1\!\left({{\gamma,\,2\lambda-\gamma}\atop{\lambda+1/2}}\biggm|
\frac{R-1+xt}{2R}
\right)
\label{eq:b1a}
  \\
  &=
(1-xt)^{-\gamma}
\,{}_2F_1\!\left({{\gamma/2,\,\gamma/2+1/2}\atop{\lambda+1/2}}\Biggm|
1-\left(\frac{R}{1-xt}\right)^2
%-\,\frac{(1-x^2)t^2}{(1-xt)^2}
\right),
\label{eq:b1b}
\end{align}
as equalities between power series in\/ $t$.
\end{subequations}
\end{theorem}

\begin{note}
  Version~(\ref{eq:b1b}) is Brafman's, rewritten; (\ref{eq:b1a})~follows by
  a quadratic transformation of~${}_2F_1$.  When $\gamma=-N$,
  $N=0,1,2,\dotsc$, the series terminate, and the ${}_2F_1$
  in~(\ref{eq:b1a}) is proportional to~$C_N^\lambda\left((1-xt)/R\right)$,
  by~(\ref{eq:Cdef}).
\end{note}

The ${}_2F_1$ in~(\ref{eq:b1a}) can be rewritten as an associated
Legendre (or Ferrers) function, with the aid of the
formula~(\ref{eq:Pdef}); as is true of the ${}_2F_1$ in~(\ref{eq:b1b})
if $\gamma=2\lambda-\nobreak1/2$.  It must be remembered that
$P_\nu^\mu(z),{\rm P}_\nu^\mu(z)$ are defined if, respectively,
$z\notin(-\infty,1]$ and $z\notin(-\infty,-1]\cup[1,\infty)$.  The
      rewriting yields
\begin{theorem}
\label{thm:rewritea}
For any\/ $\mu\notin\{1/2,1,3/2,\dots\},$ one has
\begin{subequations}
\begin{align}
\label{eq:rewritea1}
&\sum_{n=0}^\infty \frac{(-\nu-\mu)_n}{(1-2\mu)_n}\,C_n^{1/2-\mu}(x)t^n\\
&\qquad=
2^{-\mu}\,\Gamma(1-\mu)\,R^{\nu+\mu}
\left[
(z^2-1)^{\mu/2}\,P_\nu^\mu(z)
\right]\Bigm|_{z=\tfrac{1-xt}R}
,\nonumber
\intertext{for any $\nu\in\mathbb{C}$.  Moreover,}
&\label{eq:rewritea2}
\sum_{n=0}^\infty \frac{(1/2-2\mu)_n}{(1-2\mu)_n}\,C_n^{1/2-\mu}(x)t^n\\
&\qquad=
2^{-\mu}\,\Gamma(1-\mu)\,(1-xt)^{2\mu-1/2}
\left[
(z^2-1)^{\mu/2}\,P_{-1/4}^\mu(z)
\right]\Bigm|_{z=
%\tfrac{R^2+ (1-x^2)t^2}{(1-xt)^2}
2\left(\tfrac{R}{1-xt}\right)^2-1
}.\nonumber
\end{align}
\end{subequations}
These hold as equalities between power series in\/ $t$.  For
real\/~$z$, they hold as stated when $z>1$, and hold when\/
$z\in(-1,1)$ if the Legendre function\/ $P_\nu^\mu$ and\/
$z^2-\nobreak1$ are replaced by the Ferrers function\/
${\mathrm{P}}_\nu^\mu$ and\/ $1-\nobreak z^2$.
\end{theorem}

\begin{note}
  By examination, if $x$ is real and $t$ is real and of sufficiently small
  magnitude, then $z>1$ (implying that the Legendre case is applicable)
  in~(\ref{eq:rewritea1}) when $|x|>1$ and in~(\ref{eq:rewritea2}) when
  $|x|<1$.  Conversely, $z\in(-1,1)$ (implying that the Ferrers case is
  applicable) in~(\ref{eq:rewritea1}) when $|x|<1$ and
  in~(\ref{eq:rewritea2}) when $|x|>1$.  For the  $|x|<1$ case
  of~(\ref{eq:rewritea1}), cf.\ Thm.~3 of Ref.~\citenum{Cohl2013b}.
\end{note}

Specializing parameters $\nu,\mu$ in this theorem yields a number of
interesting identities.  As is summarized in the Appendix, the
associated Legendre function~$P_\nu^\mu$ and Ferrers function ${\rm
  P}_\nu^\mu$ can be written in~terms of elementary functions in
several cases.  They are referred~to here as the reducible case (i),
the quasi-\allowbreak{algebraic} cases (ii[a]), (ii[b]), and the
algebraic cases (iii), (iv[a]),~(iv[b]).

In the reducible case~(i), $\nu=-\mu+N$, $N=0,1,2,\dotsc$.  The
functions $P_\nu^\mu=P_{-\mu+N}^\mu$ and ${\rm P}_\nu^\mu={\rm
  P}_{-\mu+N}^\mu$ can then be expressed in~terms of the
polynomial~$C_N^{\mu-1/2}$ (see~(\ref{eq:Gcase})).  This leads to
\begin{theorem}
\label{thm:cfMiller}
  For any\/ $\lambda\in\mathbb{C}\setminus\{0,-1/2,-1,\dotsc\}$ and\/
  $N=0,1,2,\dots,$  one has
  \begin{align}
    \label{eq:G1}
    \sum_{n=0}^N \frac{(-N)_n}{(2\lambda)_n}\,C_n^\lambda(x)t^n 
    &= \frac{N!}{(2\lambda)_N} R^N\, C_N^\lambda\!\left(\frac{1-xt}{R}\right),\\
    \label{eq:G2}
    \sum_{n=0}^\infty \frac{(2\lambda+N)_n}{(2\lambda)_n}\,C_n^\lambda(x)t^n 
    &= \frac{N!}{(2\lambda)_N} R^{-2\lambda-N}\, C_N^\lambda\!\left(\frac{1-xt}{R}\right),
  \end{align}
as equalities between power series in\/ $t$.
\end{theorem}

\begin{proof}
  To obtain~(\ref{eq:G1}), substitute (\ref{eq:Gcase})
  into~(\ref{eq:rewritea1}); and for~(\ref{eq:G2}) do the same, first using
  $P_\nu^\mu=P_{-\nu-1}^\mu$ (or ${\rm P}_\nu^\mu={\rm P}_{-\nu-1}^\mu$).
\end{proof}

The finite sum identity~(\ref{eq:G1}) has been derived
Lie-theoretically by Miller (see Eq.~(4.11) of
Ref.~\citenum{Miller68c} and p.~204 of Ref.~\citenum{Miller68}), and
(\ref{eq:G2})~is also known.  The $N=0$ case of~(\ref{eq:G2}) is
of~course the defining generating function~(\ref{eq:defining}) for the
Gegenbauer polynomials.  There are also identities resembling
(\ref{eq:G1}),(\ref{eq:G2}) that come from~(\ref{eq:rewritea2}).

In the quasi-cyclic case~(ii[a]), the degree~$\nu$ is an integer, and
$P_\nu^\mu$ is therefore expressible in closed form
(see~(\ref{eq:allisaccomplished})).  For instance,
$P_{-1}^\mu(\coth\xi)=P_0^\mu(\coth\xi)$ equals
$\Gamma(1-\mu)^{-1}e^{\mu\xi}$, or equivalently
\begin{equation}
  P_{-1}^\mu(z) = P_0^\mu(z) = \Gamma(1-\mu)^{-1}\,
\left[(z+1)/(z-1)\right]^{\mu/2}.
\end{equation}
Setting $\nu=-1,0$ in~(\ref{eq:rewritea1}) accordingly yields the pair
\begin{subequations}
\begin{align}
\label{eq:alt1}
\sum_{n=0}^\infty \frac{(\lambda+1/2)_n}{(2\lambda)_n}\,C_n^\lambda(x)t^n
&=R^{-1}\left(\frac{1+R-xt}2\right)^{1/2-\lambda},\\
\label{eq:alt2}
\sum_{n=0}^\infty \frac{(\lambda-1/2)_n}{(2\lambda)_n}\,C_n^\lambda(x)t^n
&=\left(\frac{1+R-xt}2\right)^{1/2-\lambda}.
\end{align}
\end{subequations}
The identity~(\ref{eq:alt1}) is a well-known alternative generating
function for the Gegenbauer polynomials.  But its companion~(\ref{eq:alt2})
is less well known, though a generalization to Jacobi polynomials was found
by Carlitz;\cite{Carlitz61} as is the fact that a closed form can be
computed whenever the coefficient $(\lambda+\nobreak 1/2)_n/\allowbreak
(2\lambda)_n$ is replaced by $(\lambda+\nobreak k+\nobreak
1/2)_n/\allowbreak(2\lambda)_n$, with~$k\in\mathbb{Z}$.

In the quasi-dihedral case (ii[b]), the order~$\mu$ is a
half-odd-integer, which in (\ref{eq:rewritea1}),(\ref{eq:rewritea2})
means that the Gegenbauer parameter $\lambda=1/2-\nobreak\mu$ must be
taken to be an integer.  This is the fairly straightforward
trigonometric (e.g., Chebyshev) case, and the resulting identities are
not given here.

The focus here is on the octahedral case~(iii), when
$(\nu,\mu)\in\allowbreak\mathbb{Z}^2+\allowbreak(\pm1/6,\pm1/4)$, and
the tetrahedral subcases (iv[a]) and~(iv[b]), when
$(\nu,\mu)\in\allowbreak\mathbb{Z}^2+\allowbreak(\pm1/4,\pm1/3)$
and~$(\pm1/6,\pm1/3)$.  The functions $P_{\nu}^\mu(z)$, ${\rm
  P}_{\nu}^\mu(z)$ are then algebraic in~$z$ (see the Appendix).

\begin{theorem}
  \label{thm:algebraics}
  The Gegenbauer generating function
  \begin{displaymath}
    \sum_{n=0}^\infty \frac{(\gamma)_n}{(2\lambda)_n}\,C_n^\lambda(x)t^n
  \end{displaymath}
is algebraic\/ {\rm(1)} if\/ $\lambda\in\mathbb{Z}\pm1/4$ with\/
$\gamma-\lambda\in\mathbb{Z}\pm1/3${\rm;} or {\rm(2)}~if\/
$\lambda\in\mathbb{Z}\pm1/6$ with\/
$\gamma-\lambda\in\mathbb{Z}\pm\{1/3,1/4\}${\rm.}
\end{theorem}

\begin{proof}
  Claim~1 comes by restricting (\ref{eq:rewritea1}) to the octahedral
  case~(iii), and claim~2 by restricting (\ref{eq:rewritea1}) to the
  tetrahedral cases (iv[a]),(iv[b]).
\end{proof}

For the octahedral case (iii) and the tetrahedral case~(iv[a]), the
fundamental algebraic formulas are (\ref{subeqs:basicocto})
and~(\ref{subeqs:basictetra}), where $P_{-1/6}^{\pm1/4}(\cosh\xi)$, ${\rm
  P}_{-1/6}^{\pm1/4}(\cos\theta)$ and $P_{-1/4}^{\pm1/3}(\coth\xi)$, ${\rm
  P}_{-1/4}^{\pm1/3}(\tanh\xi)$ are given in~terms of trigonometric
functions: hyperbolic ones of~$\xi$ and circular ones of~$\theta$.
In~effect, they are given in~terms of $e^\xi$ or~$e^{{\rm i}\theta}$.  But
in Theorem~\ref{thm:rewritea}, the Legendre/Ferrers argument~$z$ equals
$(1-\nobreak xt)/R$ (in~(\ref{eq:rewritea1})) or $2\left[R/(1-\nobreak
  xt)\right]^2-1$ (in~(\ref{eq:rewritea2})).  The following table adapts
(\ref{subeqs:basicocto}),(\ref{subeqs:basictetra}) to the needs of
Theorem~\ref{thm:rewritea}.

\bigskip
\begin{tabular}{lcl}
\hline
\hline
\multicolumn{2}{l}{$z=(1-xt)/R$} & \\
\hline
$z = \cosh\xi$ & $\implies$ & $e^\xi = R^{-1}\bigl[1-(x-\sqrt{x^2-1})t\bigr]$\\
$z = \cos\theta$ & $\implies$ & $e^{{\rm i}\theta} = R^{-1}\bigl[1-(x-{\rm i}\,\sqrt{1-x^2})t\bigr]$\\
$z = \coth\xi$ & $\implies$ & $e^\xi=t\sqrt{x^2-1}/(1-R-xt)$\\
$z = \tanh\xi$ & $\implies$ & $e^\xi=t\sqrt{1-x^2}/(-1+R+xt)$\\
\hline
\hline
\multicolumn{2}{l}{$z=2\left[R/(1-xt)\right]^2-1$} & \\
\hline
$z = \cosh\xi$ & $\implies$ & $e^{\xi/2} = (1-xt)/\!\left[R-t\sqrt{1-x^2}\right]$\\
$z = \cos\theta$ & $\implies$ & $e^{{\rm i}\theta/2} = (1-xt)/\!\left[R- {\rm i}\,t\sqrt{x^2-1}\right]$\\
$z = \coth\xi$ & $\implies$ & $e^\xi = R/\left(t\sqrt{1-x^2}\right)$\\
$z = \tanh\xi$ & $\implies$ & $e^\xi = R/\left(t\sqrt{x^2-1}\right)$\\
\hline
\hline
\end{tabular}
\bigskip

By combining the $(\nu,\mu)=(-1/6,1/4)$ case of~(\ref{eq:rewritea1})
with~(\ref{eq:basicocto1}), aided by the first line of this table, one
readily derives an explicit, octahedrally algebraic generating
function for the set of polynomials $\{C_n^{1/4}(x)\}_{n=0}^\infty$.
It appeared in the Introduction as Eq.~(\ref{eq:C14example}).  One
also derives a tetrahedrally algebraic generating function for
$\{C_n^{1/6}(x)\}_{n=0}^\infty$, namely
\begin{align}
&\sum_{n=0}^\infty \frac{(-1/12)_n}{(1/3)_n}\,C_n^{1/6}(x)t^n\label{eq:C16example}\\
\nonumber
&=2^{-7/12}\,3^{-3/8}R^{1/12}
(\sinh\xi)^{-1/3}
\left[ \sqrt{\sqrt{3}+1}\:f_+ + \sqrt{\sqrt{3}-1}\:f_- \right],\\
&\qquad e^\xi \defeq t\sqrt{x^2-1}/(1-R-xt)\:;
\nonumber\\
&=2^{-7/12}\,3^{-3/8}R^{1/12}
(\cosh\xi)^{-1/3}
\left[ \sqrt{\sqrt{3}+1}\:g_+ + \sqrt{\sqrt{3}-1}\:g_- \right],\nonumber\\
&\qquad e^\xi \defeq t\sqrt{1-x^2}/(-1+R+xt).
\nonumber
\end{align}
This comes by combining the $(\nu,\mu)=(-1/4,1/3)$ case
of~(\ref{eq:rewritea1}) with (\ref{eq:basictetra1})
and~(\ref{eq:basictetra2}), aided by the third and fourth lines of the
table.  Here, the functions $f_\pm=f_\pm(\coth\xi)$ and
$g_\pm=g_\pm(\tanh(\xi))$ are algebraic in~$e^\xi$ and are defined in
(\ref{eq:lastminutea}),(\ref{eq:lastminuteb}).  The two right-hand sides
of~(\ref{eq:C16example}) are equivalent, but when the argument~$x$ is real,
they are most useful when, respectively, $|x|>1$ and~$|x|<1$.  Additional
explicitly algebraic generating functions that arise from
(\ref{eq:rewritea1}) or~(\ref{eq:rewritea2}) can be worked~out.

\smallskip
The preceding results stemmed from Theorem~\ref{thm:bgf1} but can be
generalized, because Brafman's first generating function can be
hypergeometrically extended.  The extension uses the identities
appearing in his Ref.~\citenum{Brafman57} as Eqs.\ (4) and~(11).  The
latter identity, which is a specialization of the former to Gegenbauer
polynomials, can be restated as~follows.

\begin{lemma}
\label{lem:key}
  For any\/ $\lambda\in\mathbb{C}$, and parameters\/ $c_1,\dots,c_p$,
  $d_1,\dots,d_q$ and\/ $u\in\mathbb{C}$ for which the below\/
  ${}_{p+1}F_q$ coefficients are defined, one has
  \begin{multline*}
    \sum_{n=0}^\infty {}_{p+1}F_q\!
    \left(
    {{-n,\,c_1,\dots,\,c_p}\atop{d_1,\dots, d_q}}\biggm | u
    \right)C_n^\lambda(x)t^n\\
    =
    R^{-2\lambda} \sum_{n=0}^\infty 
    \frac{(c_1)_n\dotsb(c_p)_n}{(d_1)_n\dotsb(d_q)_n}\,C_n^\lambda\!\left(\frac{x-t}R\right) \left(\frac{-tu}R\right)^n,
  \end{multline*}
  as an equality between power series in\/ $t$.
\end{lemma}

\begin{note}
  In Ref.~\citenum{Brafman57}, the argument of the $C_n^\lambda$ on the
  right is written as~$w$, which is defined in Eq.~(1) of that work to
  equal $2(x-t)/R$.  The~`2' is easily seen to be erroneous, and has been
  removed.  This identity is proved by series rearrangement, once each
  $C_n^\lambda$ has been expressed hypergeometrically: not as
  in~(\ref{eq:Cdef}), but in a form that incorporates a quadratic
  transformation.
\end{note}

By applying the $p=q=1$ case of this lemma to the statement of
Theorem~\ref{thm:bgf1}, one readily obtains the following corollaries of
(\ref{eq:b1a}) and~(\ref{eq:b1b}).  Here and below, $U$~signifies
$[1-2(1-u)xt + (1-u)^2t^2]^{1/2}$, with $U=1$ when~$t=0$.  Hence,
$U$~interpolates between $R$ at~$u=0$ and unity at~$u=1$.  Similarly,
$R^2+\nobreak u(x-\nobreak t)t$ interpolates between $R^2$ and $1-\nobreak
xt$.  It should be noted that to make (\ref{eq:b1bextended})
resemble~(\ref{eq:b1b}) as closely as possible, Euler's transformation
of~${}_2F_1$ has been applied to its right-hand side.
\begin{theorem}
\label{thm:bgf1extended}
  For any\/ $\lambda\in\mathbb{C}\setminus\{0,-1/2,-1,\dotsc\}$ and
  $\gamma\in\mathbb{C},$ and arbitrary\/~$u,$ one has
\begin{subequations}
\begin{align}
&  \sum_{n=0}^\infty {}_2F_1\!\left(
{{-n,\,2\lambda-\gamma}\atop{2\lambda}}
\biggm| u
\right)
C_n^\lambda(x)t^n\nonumber\\
&  =
U^{\gamma-2\lambda}
  R^{-\gamma}\,{}_2F_1\!\left({{\gamma,\,2\lambda-\gamma}\atop{\lambda+1/2}}\biggm|
\frac{UR-[R^2+u(x-t)t]}{2\,UR}
\right)
\label{eq:b1aextended}
  \\
&  =
U^{2\gamma-2\lambda} [R^2+u(x-t)t]^{-\gamma}
\,{}_2F_1\!\left({{\gamma/2,\,\gamma/2+1/2}\atop{\lambda+1/2}}\Biggm|
%-\,\frac{u^2(1-x^2)t^2}{[R^2+u(x-t)t]^2}
1-\left[\frac{UR}{R^2+u(x-t)t}\right]^2
\right),
\label{eq:b1bextended}
\end{align}
where the equalities are between power series in\/ $t$.
\end{subequations}
\end{theorem}

\begin{note}
  When $\gamma=-N$, $N=0,1,2,\dotsc$, the right-hand series terminate, and
  the ${}_2F_1$ in~(\ref{eq:b1aextended}) is proportional
  to~$C_N^\lambda\left([R^2+u(x-t)t]/UR\right)$.
\end{note}

Theorem~\ref{thm:bgf1extended} is not merely a corollary of
Theorem~\ref{thm:bgf1}, but an extension.  It reduces to
Theorem~\ref{thm:bgf1} when $u=1$, because the left-hand ${}_2F_1$
then equals $(\gamma)_n/(2\lambda)_n$, by the Chu--Vandermonde
formula.  Rewriting each right-hand ${}_2F_1$ in
Theorem~\ref{thm:bgf1extended} as an associated Legendre function
yields the following, which is an extension of
Theorem~\ref{thm:rewritea}.  (In~(\ref{eq:b1bextended}), the rewriting
is possible only if $\gamma=\allowbreak2\lambda-\nobreak1/2$.)

\begin{theorem}
\label{thm:rewriteaextended}
For any\/ $\mu\notin\{1/2,1,3/2,\dots\},$ and arbitrary\/~$u,$ one has
\begin{subequations}
\begin{align}
\label{eq:rewritea1extended}
&\sum_{n=0}^\infty {}_2F_1\!\left(
{{-n,\,\nu-\mu+1}\atop{1-2\mu}}
\Bigm| u
\right)
C_n^{1/2-\mu}(x)t^n\\
&=
2^{-\mu}\,\Gamma(1-\mu)\,
U^{-\nu+\mu-1}
R^{\nu+\mu}
\left[
(z^2-1)^{\mu/2}\,P_\nu^\mu(z)
\right]\Bigm|_{z=\tfrac{R^2+u(x-t)t}{UR}},
\nonumber\\
\intertext{for any $\nu\in\mathbb{C}$.  Moreover,}
\label{eq:rewritea2extended}
&\sum_{n=0}^\infty {}_2F_1\!\left(
{{-n,\,1/2}\atop{1-2\mu}}
\Bigm| u
\right)
C_n^{1/2-\mu}(x)t^n\\
&=
2^{-\mu}\,\Gamma(1-\mu)\,U^{-2\mu}\nonumber\\
&\quad{}\times\left[R^2+u(x-t)t\right]^{2\mu-1/2}
\left[
(z^2-1)^{\mu/2}\,P_{-1/4}^\mu(z)
\right]\Bigm|_{z=
2\left[\tfrac{UR}{R^2+u(x-t)t}\right]^2-1
}.
\nonumber
\end{align}
\end{subequations}
These hold as equalities between power series in\/ $t$.  For real\/~$z$,
they hold as stated when $z>1$, and hold when\/ $z\in(-1,1)$ if\/
$P_\nu^\mu$ and\/ $z^2-\nobreak1$ are replaced by\/ ${\mathrm{P}}_\nu^\mu$
and\/ $1-\nobreak z^2$.
\end{theorem}

By specializing Theorem~\ref{thm:rewriteaextended}, one can immediately
extend the preceding results on closed-form generating functions and
algebraicity to include the free parameter~$u$.  The following
specializations of~(\ref{eq:rewritea1extended}), based on~(\ref{eq:Gcase}),
are extensions of the identities (\ref{eq:G1}) and~(\ref{eq:G2}) of
Theorem~\ref{thm:cfMiller}; for the latter, cf.\ Eq.~(4.14) of
Viswanathan~\cite{Viswanathan68} and Eq.~(5.121) of Miller.~\cite{Miller68}

\begin{theorem}
\label{thm:cfMillerextended}
  For any\/ $\lambda\in\mathbb{C}\setminus\{0,-1/2,-1,\dotsc\}$ and\/
  $N=0,1,2,\dots,$ and arbitrary\/ $u,$  one has
  \begin{align*}
    &\sum_{n=0}^\infty {}_2F_1\!\left( {{-n,\,2\lambda+N}\atop{2\lambda}} \biggm| u \right) C_n^\lambda(x)t^n\\
    &\qquad\qquad\qquad\qquad\qquad= \frac{N!}{(2\lambda)_N} U^{-2\lambda-N}   R^N\, C_N^\lambda\!\left(\frac{R^2+u(x-t)t}{UR}\right),\\
    &\sum_{n=0}^\infty {}_2F_1\!\left( {{-n,\,-N}\atop{2\lambda}} \biggm| u \right) C_n^\lambda(x)t^n \\
    &\qquad\qquad\qquad\qquad\qquad= \frac{N!}{(2\lambda)_N} U^{N} R^{-2\lambda-N}\, C_N^\lambda\!\left(\frac{R^2+u(x-t)t}{UR}\right),
  \end{align*}
as equalities between power series in\/ $t$.
\end{theorem}

A $u$-dependent extension of Theorem~\ref{thm:algebraics}, which was
implied by~(\ref{eq:rewritea1}), can also be obtained; it is an
immediate consequence of the $u$\nobreakdash-dependent
extension~(\ref{eq:rewritea1extended}) of~(\ref{eq:rewritea1}), and is
\begin{theorem}
  \label{thm:algebraicsextended}
  The Gegenbauer generating function
  \begin{displaymath}
    \sum_{n=0}^\infty {}_2F_1\!\left(
    {{-n,\,2\lambda-\gamma}\atop{2\lambda}}
    \Bigm| u
    \right)
    C_n^\lambda(x)t^n
  \end{displaymath}
is algebraic, for arbitrary\/ $u,$ 
{\rm(1)} if\/ $\lambda\in\mathbb{Z}\pm1/4$ with\/
$\gamma-\lambda\in\mathbb{Z}\pm1/3${\rm;} or {\rm(2)}~if\/
$\lambda\in\mathbb{Z}\pm1/6$ with\/
$\gamma-\lambda\in\mathbb{Z}\pm\{1/3,1/4\}${\rm.}
\end{theorem}

Explicit expressions for these $u$-dependent generating functions, which
are algebraic in~$u$ as~well as in~$t,x$, can be computed with some effort
from Theorem~\ref{thm:rewriteaextended}, if one exploits the fundamental
formulas (\ref{subeqs:basicocto}),(\ref{subeqs:basictetra}).
% One example is FOO, which reduces to Eq.~XXX when $u=1$.

\section{The Second Gegenbauer Generating Function}
\label{sec:second}

The following theorem presents Brafman's second ${}_2F_1$-based generating
function.\cite{Brafman51}
\begin{theorem}
\label{thm:bgf2}  
  For any\/ $\lambda\in\mathbb{C}\setminus\{0,-1/2,-1,\dotsc\}$ and\/
  $\gamma\in\mathbb{C},$ one has
\begin{subequations}
\begin{align}
  &\sum_{n=0}^\infty \frac{(\gamma)_n(2\lambda-\gamma)_n}{(2\lambda)_n(\lambda+1/2)_n}\,C_n^\lambda(x)t^n\nonumber\\
  &=
{}_2F_1\!\left({{\gamma,\,2\lambda-\gamma}\atop{\lambda+1/2}}\biggm|\frac{1-R-t}{2}\right)
{}_2F_1\!\left({{\gamma,\,2\lambda-\gamma}\atop{\lambda+1/2}}\biggm|\frac{1-R+t}{2}\right)
\label{eq:b2a}
\\
  &=(1-2xt)^{-\gamma}\nonumber\\
&\qquad{}\times 
{}_2F_1\!\left({{\gamma/2,\,\gamma/2+1/2}\atop{\lambda+1/2}}\biggm|1-\,\frac{1}{(R+t)^2}\right)\nonumber\\
&\qquad{}\times 
{}_2F_1\!\left({{\gamma/2,\,\gamma/2+1/2}\atop{\lambda+1/2}}\biggm|1-\,\frac{1}{(R-t)^2}\right),
\label{eq:b2b}
\end{align}
as equalities between power series in\/ $t$.
\end{subequations}
\end{theorem}

\begin{note}
  Version~(\ref{eq:b2a}) is Brafman's; (\ref{eq:b2b})~follows by
  quadratically transforming each~${}_2F_1$.  When $\gamma=-N$,
  $N=0,1,2,\dotsc$, the series terminate, and the right-hand side
  of~(\ref{eq:b2a}) is proportional to $C_N^\lambda(R+t)C_N^\lambda(R-t)$.
  Irrespective of~$\gamma$, the Legendre (i.e., $\lambda=1/2$) case is of
  special interest.  The $C_n^\lambda$ then reduces to the Legendre
  polynomial~$P_n$, and each ${}_2F_1$ in~(\ref{eq:b2a}) is proportional to
  the Legendre function~$P_{-\gamma}$.
\end{note}

By applying the $p=q=2$ case of Lemma~\ref{lem:key} to the statement of
Theorem~\ref{thm:bgf2}, one readily obtains the following corollaries of
(\ref{eq:b2a}) and~(\ref{eq:b2b}); for the former,
cf.~Brafman.\cite{Brafman57}

\begin{theorem}
\label{thm:bgf2extended}
  For any\/ $\lambda\in\mathbb{C}\setminus\{0,-1/2,-1,\dotsc\}$ and\/
  $\gamma\in\mathbb{C},$ and arbitrary\/~$u,$ one has
\begin{subequations}
\begin{align}
  &\sum_{n=0}^\infty {}_3F_2\!\left({{-n,\,\gamma,\,2\lambda-\gamma}\atop{2\lambda,\,\lambda+1/2}}\biggm|u\right)C_n^\lambda(x)t^n\nonumber\\
  &= R^{-2\lambda}\,
  {}_2F_1\!\left({{\gamma,\,2\lambda-\gamma}\atop{\lambda+1/2}}\biggm|\frac{R-U+ut}{2R}\right)
{}_2F_1\!\left({{\gamma,\,2\lambda-\gamma}\atop{\lambda+1/2}}\biggm|\frac{R-U-ut}{2R}\right)
\label{eq:b2aextended}
\\
  &=(U^2 - u^2t^2)^{-\gamma}R^{2\gamma-2\lambda}\nonumber\\
&\qquad{}\times 
{}_2F_1\!\left({{\gamma/2,\,\gamma/2+1/2}\atop{\lambda+1/2}}\biggm|1-\,\frac{R^2}{(U-ut)^2}\right)\nonumber\\
&\qquad{}\times 
{}_2F_1\!\left({{\gamma/2,\,\gamma/2+1/2}\atop{\lambda+1/2}}\biggm|1-\,\frac{R^2}{(U+ut)^2}\right),
\label{eq:b2bextended}
\end{align}
where the equalities are between power series in\/ $t$.
\end{subequations}
\end{theorem}

\begin{note}
  When $\gamma=-N$, $N=0,1,2,\dots$, the series terminate, and the
  right-hand side of~(\ref{eq:b2aextended}) is proportional to
  $C_N^\lambda((U-\nobreak ut)/\nobreak R)\allowbreak
  C_N^\lambda((U+\nobreak ut)/\nobreak R)$.  

  The Legendre case of~(\ref{eq:b2aextended}), i.e., that of
  $\lambda=1/2$ with $\gamma$~unrestricted, was derived before Brafman
  by Rice~\cite{Rice40}, who attributed its $u=1$ sub-case to
  Bateman.\cite{Bateman38}\ \ (Cf.\ Rice's Eqs.\ (2.11) and (2.14),
  and Bateman's~(4.3).)
\end{note}

Theorem~\ref{thm:bgf2extended} does not reduce easily to
Theorem~\ref{thm:bgf2} (for instance, by setting $u=1$, which reduces
Theorem~\ref{thm:bgf1extended} to Theorem~\ref{thm:bgf1}).  It is
better described as a corollary than as an extension.

But the ${}_2F_1$ functions in Theorems \ref{thm:bgf2}
and~\ref{thm:bgf2extended} are familiar: they have the same parameters
as in Theorems \ref{thm:bgf1} and~\ref{thm:bgf1extended}, to which
Theorems \ref{thm:bgf2} and~\ref{thm:bgf2extended} are analogous.
Rewriting Theorem~\ref{thm:bgf2} in~terms of associated Legendre or
Ferrers functions yields the following, which is analogous to
Theorem~\ref{thm:rewritea}.  (For~(\ref{eq:15}), cf.\ Thm.~2 of Cohl
and MacKenzie.\cite{Cohl2013b}.)
\begin{theorem}
For any\/ $\mu\notin\{1/2,1,3/2,\dotsc\},$ one has
\begin{subequations}
\begin{align}
&\sum_{n=0}^\infty
\frac{(-\nu-\mu)_n (1+\nu-\mu)_n}{(1-2\mu)_n(1-\mu)_n}\,C_n^{1/2-\mu}(x)t^n\label{eq:15}
\\
&\qquad\qquad\qquad\qquad\qquad\qquad\quad=2^{-2\mu}\,\Gamma(1-\mu)^2 \,\mathcal{F}_\nu^\mu(R+t)\,\mathcal{F}_\nu^\mu(R-t),\nonumber
\end{align}
for any\/ $\nu\in\mathbb{C}$.
Moreover,
\begin{sizealign}{\small}
&\sum_{n=0}^\infty
\frac{(1/2-2\mu)_n (1/2)_n}{(1-2\mu)_n(1-\mu)_n}\,C_n^{1/2-\mu}(x)t^n
\\
&=2^{-2\mu}\,\Gamma(1-\mu)^2 \,(1-2xt)^{2\mu-1/2} \,\mathcal{F}_{-1/4}^\mu\!\left(\frac2{(R-t)^2}-1\right)\,\mathcal{F}_{-1/4}^\mu\!\left(\frac2{(R+t)^2}-1\right).\nonumber
\end{sizealign}
\end{subequations}
These hold as equalities between power series in\/ $t$.  The function
$\mathcal{F}_\nu^\mu(z)$ is defined as\/
$(z^2-\nobreak1)^{\mu/2}\,P_\nu^\mu(z)$ or\/ $(1-z^2)^{\mu/2}\,{\rm
  P}_\nu^\mu(z)$, which are equivalent.  {\rm(}For real\/~$z$, the Legendre
definition should be used if\/ $z>1$ and the Ferrers definition if\/
$z\in(-1,1)$.{\rm)}
\end{theorem}

The rewriting of Theorem~\ref{thm:bgf2extended} in terms of associated
Legendre functions proceeds similarly to the rewriting of
Theorem~\ref{thm:bgf2}.  By specializing parameters in the two
rewritten theorems, one can derive a number of interesting identities.
For example, in the reducible case~(i), when $\nu=-\mu+\nobreak N$,
$N=0,1,2,\dotsc,$ one can derive analogues of Theorems
\ref{thm:cfMiller} and~\ref{thm:cfMillerextended}.

The focus here is on the octahedral and tetrahedral algebraic cases.  From
the algebraic formulas in the Appendix, substituted into the rewritten
theorems, one deduces the following from their first halves; again, as in
the last section.

\begin{theorem}
  The Gegenbauer generating functions
  \begin{displaymath}
    \sum_{n=0}^\infty
    \frac{(\gamma)_n(2\lambda-\gamma)_n}{(2\lambda)_n(\lambda+1/2)_n}\,
   C_n^\lambda(x)t^n,
    \qquad
    \sum_{n=0}^\infty 
        {}_3F_2\!\left({{-n,\,\gamma,\,2\lambda-\gamma}\atop{2\lambda,\,\lambda+1/2}}\biggm|u\right)C_n^\lambda(x)t^n,
  \end{displaymath}
are algebraic 
{\rm(1)} if\/ $\lambda\in\mathbb{Z}\pm1/4$ with\/
$\gamma-\lambda\in\mathbb{Z}\pm1/3${\rm;} or {\rm(2)}~if\/
$\lambda\in\mathbb{Z}\pm1/6$ with\/
$\gamma-\lambda\in\mathbb{Z}\pm\{1/3,1/4\}${\rm.}
In the latter generating
function, $u$~is arbitrary: the algebraicity is in\/ $x,t$ and\/~$u$.
\end{theorem}

\section{The Poisson Kernel}
\label{sec:poisson}

The Poisson kernel for a set of orthogonal polynomials
$\{p_n(x)\}_{n=0}^\infty$ plays a major role in approximation theory.  It
is a bilinear generating function of the form
\begin{equation}
  K_t(x,y) = \sum_{n=0}^\infty h_n\,p_n(x)p_n(y)\,t^n,
\end{equation}
where the normalization coefficients~$h_n$ would be absent if the
polynomials were orthonormal and not merely orthogonal.  The Poisson
kernel for the Gegenbauer polynomials can be expressed in~terms
of~${}_2F_1$, as can a slightly simpler companion
function.\cite{Weisner55,Gasper83,Koornwinder2015}\ \ Let
$x=\cos\theta$ and $y=\cos\phi$, which are appropriate when
$x,y\in(-1,1)$, and define
\begin{subequations}
\begin{align}
\tilde z &= \frac{-4\,t\sin\theta \sin\phi}{1-2t\cos(\theta-\phi) + t^2},
\\
z & =  \frac{4\,t^2\sin^2\theta \sin^2\phi}{(1-2t\cos\theta\cos\phi + t^2)^2},
\end{align}
\end{subequations}
which are related by $z=[\tilde z/(2-\tilde z)]^2$.  The Poisson kernel for
$\{C_n^\lambda(x)\}_{n=0}^\infty$ is
\begin{align}
&\sum_{n=0}^\infty \frac{\lambda+n}{\lambda}\,\frac{n!}{(2\lambda)_n}\,C_n^\lambda(x)C_n^\lambda(y)\,t^n\label{eq:Pk}\\
&\qquad\qquad=\frac{1-t^2}{\left[1-2t\cos(\theta-\phi) + t^2\right]^{\lambda+1}}\:{}_2F_1\!\left({{\lambda,\,\lambda+1}\atop{2\lambda}} \Bigm|\tilde z\right),\nonumber\\
&\qquad\qquad=\frac{1-t^2}{(1-2t\cos\theta\cos\phi + t^2)^{\lambda+1}}\:{}_2F_1\!\left({{(\lambda+1)/2,\,(\lambda+2)/2}\atop{\lambda+1/2}} \Bigm|z\right)\nonumber\\
\intertext{and its companion is}
&\sum_{n=0}^\infty\frac{n!}{(2\lambda)_n}\,C_n^\lambda(x)C_n^\lambda(y)\,t^n\label{eq:Pkcompanion}\\
&\qquad\qquad=\frac{1}{\left[1-2t\cos(\theta-\phi) + t^2\right]^{\lambda}}\:{}_2F_1\!\left({{\lambda,\,\lambda}\atop{2\lambda}}\Bigm|\tilde z\right).\nonumber\\
&\qquad\qquad=\frac{1}{(1-2t\cos\theta\cos\phi + t^2)^{\lambda}}\:{}_2F_1\!\left({{\lambda/2,\,(\lambda+1)/2}\atop{\lambda+1/2}} \Bigm|z\right)\nonumber
\end{align}
In each of (\ref{eq:Pk}) and~(\ref{eq:Pkcompanion}), the two right-hand
sides are related by a quadratic hypergeometric transformation.
Equation~(\ref{eq:Pk}) can be obtained from Eq.~(\ref{eq:Pkcompanion}) by
applying the operator $\lambda^{-1}t^{-\lambda+1} \frac{{\rm d}}{{\rm d}t}
\circ t^{\lambda}$ to both sides.

When the parameter~$\lambda$ is an integer (e.g., in the Chebyshev case),
the Poisson kernel and its companion are elementary functions.  When
$\lambda=1/2$, so that $\{C_n^\lambda(x)\}_{n=0}^\infty$ are the Legendre
polynomials, Watson\cite{Watson32c} expressed the companion in~terms of the
first complete elliptic integral function, $K=K(m)$.  In principle, this
can be done when $\lambda$~is any half-odd-integer.

It does not seem to have been remarked that when $\lambda$~differs
from an integer by one-fourth or one-sixth, expressions in~terms of
complete elliptic integrals can also be obtained.  This is implied by
the pattern of hypergeometric parameters in~(\ref{eq:Pkcompanion}), as
will be explained.  The focus here is on the cases $\lambda=1/4$
and~$\lambda=1/6$; the general $\lambda\in\mathbb{Z}\pm1/4$ and
$\lambda\in\mathbb{Z}\pm1/6$ cases can be handled by applying the
contiguity relations of~${}_2F_1$.

The Gauss hypergeometric ODE satisfied by the function
${}_2F_1(a,b;c;z)$ has singular points at $z=0,1,\infty$, with
respective characteristic exponents $0,1-\nobreak c$; $0,c-\nobreak
a-\nobreak b$; and~$a,b$.  The respective exponent \emph{differences}
are $1-\nobreak c$, $c-\nobreak a-\nobreak b$, and $b-\nobreak a$, and
differences are significant only up to sign.  It is well known that if
a Gauss ODE has an unordered set of (up-to-sign) exponent differences
$\{\frac12,\delta_1,\delta_2\}$, it admits a quadratic transformation
to one with differences $\{\delta_1,\delta_1,2\delta_2\}$, and the
solutions of the two ODEs will correspond.  (For instance, in each of
(\ref{eq:Pk}) and~(\ref{eq:Pkcompanion}) the first right-hand side
comes from the second in this way.)  One may write
$\{\frac12,\delta_1,\delta_2\}\sim\allowbreak
\{\delta_1,\delta_1,2\delta_2\}$.  There are hypergeometric
transformations of higher order than the quadratic (see Sec.~25 of
Ref.~\citenum{Poole36}, \emph{inter alia}).  In particular, there are
sextic ones that arise as compositions of quadratic and cubic ones,
the action of which is summarized by
$\{\frac12,\frac13,\delta\}\sim\allowbreak
\{\frac13,\frac13,2\delta\}\sim\allowbreak
\{2\delta,2\delta,2\delta\}$.

The exponent differences for the ${}_2F_1$'s in the two right-hand
sides of~(\ref{eq:Pkcompanion}) are respectively
$1-\nobreak2\lambda,0,0$ and $1/2-\nobreak\lambda,0,1/2$.  If
$\lambda=1/6$, the second of these triples (when unordered) is
$\{\frac12,\frac13,0\}\sim\allowbreak \{0,0,0\}$.  If $\lambda=1/4$,
the first is $\{\frac12,0,0\}\sim\allowbreak \{0,0,0\}$.  In other
words, a sextic and a quadratic transformation will respectively
convert the ${}_2F_1$'s in the two right-hand sides
of~(\ref{eq:Pkcompanion}) to solutions of a Gauss ODE with exponent
differences $0,0,0$.  This is the Gauss ODE with parameters $a=b=1/2$,
$c=1$, one of the solutions of which is ${}_2F_1(1/2,1/2;1;z')$.
(Here, $z'$~signifies the new independent variable, which is
determined by the hypergeometric transformation.)  But
${}_2F_1(1/2,1/2;1;z')$ equals $(2/\pi)K(z')$, and the full solution
space of the ODE is spanned by $K(z')$ and~$K'(z')\defeq
K(1-\nobreak z')$.

The expressions resulting from the just-described reduction procedure
are somewhat inelegant, but better ones can be obtained heuristically.
In~fact, one can begin with the Poisson kernel itself, rather than its
companion.  The case $\lambda=1/4$ is illustrative.  When
$\lambda=1/4$, the~${}_2F_1$ in the first right-hand side
of~(\ref{eq:Pk}) is of the form ${}_2F_1(1/4,5/4;1/2;w)$, where
$w=\tilde z$.  An explicit formula for this function can be found in
the database of Roach\cite{Roach96}, which is currently available at
\texttt{www.planetquantum.com}.  It is
\begin{align}
&\frac{\Gamma(1/4)^2}{2\sqrt{\pi}}\,{}_2F_1\left(
{{1/4,\,5/4}\atop{1/2}}\biggm| w
\right) \\
&\qquad=\frac{2\sqrt{w}}{1-w}\, E(\tilde w_+)
-\frac{2\sqrt{w}}{1-w}\, E(\tilde w_-)
+\frac1{1+\sqrt{w}}\, K(\tilde w_+)
+\frac1{1-\sqrt{w}}\, K(\tilde w_-),
\nonumber
\end{align}
where $\tilde w_\pm = (1\pm\sqrt{w})/2$.  The presence of the second
complete elliptic integral, $E=E(m)$, for which the exponent differences
are $1,1,0$ and not $0,0,0$, can be attributed to an application of the
contiguous relations of~${}_2F_1$.

\section{Final Remarks}
\label{sec:final}

It has been shown that if the Gegenbauer parameter~$\lambda$ differs by
one-fourth or one-sixth from an integer, there are several generating
functions for the Gegenbauer polynomials $\{C_n^\lambda\}_{n=0}^\infty$
that are algebraic.  These are special cases of Brafman's generating
functions, including the extended ones with an additional free parameter
(denoted~$u$ here).  Gegenbauer polynomials with $\lambda$ restricted as
stated were shown to be special in another way: the Poisson kernel computed
from them can be expressed with the aid of hypergeometric transformations
in~terms of complete elliptic integrals.

The results on algebraicity are consequences of Schwarz's
classification of the algebraic cases of the Gauss function~${}_2F_1$,
and the explicit examples of algebraic generating functions came from
recently developed closed-form expressions for certain algebraic
${}_2F_1$'s with octahedral and tetrahedral monodromy; i.e.,
octahedral and tetrahedral associated Legendre
functions.~\cite{Maier24} This is because the ${}_2F_1$'s in Brafman's
generating functions admit quadratic transformations, so that in
essence, they are Legendre functions.

Generalizations can be considered.  One matter worthy of investigation
is the relevance of icosahedral~${}_2F_1$'s, which though algebraic
cannot be expressed in~terms of radicals.  Some parametric formulas
for them are known\cite{Vidunas2013}, and may yield manageable
parametrizations of the consequent algebraic generating functions for
Gegenbauer polynomials.

The generalization from Gegenbauer to Jacobi polynomials is also worth
pursuing.  It follows readily from Schwarz's classification that
certain special cases of Brafman's second generating function,
generalized to non-Gegenbauer Jacobi polynomials but still expressed
in~terms of~${}_2F_1$, are algebraic functions of their arguments.
However, the Jacobi-polynomial generalization of his first generating
function is known to involve the Appell function~$F_4$, i.e., a
bivariate hypergeometric function.  A full classification of the
algebraic cases of the first generating function, generalized to
Jacobi polynomials, will require results on the algebraicity of Appell
functions; and the same is true of the Poisson kernel.

Finally, it should be mentioned that Brafman's extension procedure,
leading to identities parametrized by~$u$, is not the only one that
can be applied to Gegenbauer generating functions.  By exploiting the
connection formula for Gegenbauer polynomials, Cohl and
collaborators\cite{Cohl2013a,Cohl2013c} have obtained novel extensions
of the defining relation~(\ref{eq:defining}), as Eq.~(11) of their
Ref.~\citenum{Cohl2013a}, and of Brafman's second
identity~(\ref{eq:b2a}), as Eq.~(26) of their
Ref.~\citenum{Cohl2013c}.  For suitably chosen parameter values, such
extensions will be algebraic.

\renewcommand{\theequation}{A.\arabic{equation}}
\setcounter{equation}{0}

\section*{Appendix. Associated Legendre Functions in Closed Form}

The associated Legendre function $P_\nu^\mu(z)$ of degree
$\nu\in\mathbb{C}$ and order $\mu\in\mathbb{C}$ is defined in~terms of the
Gauss function~${}_2F_1$ by
\begin{equation}
\label{eq:Pdef}
P_\nu^\mu(z) = \frac{2^\mu}{\Gamma(1-\mu)}
(z^2-1)^{-\mu/2}\,{}_2F_1\!\left(
{{-\nu-\mu,\, 1+\nu-\mu}\atop{1-\mu}}
\biggm|\frac{1-z}2\right).
\end{equation}
The Ferrers function ${\rm P}_\nu^\mu$ is defined similarly, with
$1-\nobreak z^2$ replacing $z^2-\nobreak1$.  By convention,
${P}_\nu^\mu(z)$ and ${\rm P}_\nu^\mu(z)$ are defined and analytic on the
complex $z$-plane, with the respective omissions of the cut $(-\infty,1]$
  and the cut-pair $(-\infty,-1]\cup[1,\infty)$.  When $\mu=1,2,\dotsc$,
      (\ref{eq:Pdef})~must be taken in a limiting sense.  In the singular
      case when $\nu=0,1,2,\dotsc$ and $\mu-\nobreak\nu$ is a positive
      integer, ${P}_\nu^\mu$~and~${\rm P}_\nu^\mu$ are identically zero.

On their respective domains, ${P}_\nu^{\pm\mu}$ and ${\rm P}_\nu^{\pm\mu}$
span the two-dimensional solution space of the associated Legendre ODE,
except when $(\nu,\mu)\in\mathbb{Z}^2$.  This space can also be viewed as
the span of $P_\nu^\mu$ and~$Q_\nu^\mu$, resp.\ ${\rm P}_\nu^\mu$ and~${\rm
  Q}_\nu^\mu$, where $Q_\nu^\mu$ and~${\rm Q}_\nu^\mu$ are the associated
Legendre and Ferrers functions of the second kind.  (Again, singular cases
are excepted.)  The function ~${P}_\nu^{\mu}(z)$ is singled~out as an
element $f(z)$ of the solution space with
\begin{equation}
\label{eq:condition}
 f(z)\sim\frac{2^{\mu/2}}{\Gamma(1-\mu)}(z-\nobreak1)^{-\mu/2},
\qquad z\to1 
\end{equation}
as asymptotic behavior.

For all $\nu,\mu\in\mathbb{C}$, it follows from~(\ref{eq:Pdef}) that
$P_{-\nu-1}^\mu = P_{\nu}^\mu$ and ${\rm P}_{-\nu-1}^\mu = {\rm
  P}_{\nu}^\mu$.  Also, the ordered pair $(\nu,\mu)$ can be displaced by
any element of~$\mathbb{Z}^2$, for either $P_\nu^\mu$ or~${\rm P}_\nu^\mu$,
by applying an appropriate differential operator.  (See Sec.~6 of
Ref.~\citenum{Maier24}.)  Such `ladder operators,' which increment and
decrement $\nu$ and/or~$\mu$, come from the contiguity relations
of~${}_2F_1$.

There are several cases when the functions $P_\nu,{\rm P}_\nu^\mu$ are
elementary; or to put it more broadly, when all solutions of the
associated Legendre ODE can be reduced to
quadratures.\cite{Kimura69}\ \ These include the case when the ODE, or
the equivalent ODE satisfied by the ${}_2F_1$ in~(\ref{eq:Pdef}), is
`reducible' (see Sec.~2.2 of Ref.~\citenum{Erdelyi53}); and certain
algebraic cases, when the ODE has a finite projective monodromy group
(see Sec.~2.7.2 of Ref.~\citenum{Erdelyi53} and Chap.~VII of
Ref.~\citenum{Poole36}).  In the algebraic cases, this group as a
subgroup of the M\"obius group may be cyclic, dihedral, octahedral,
tetrahedral, or icosahedral, but the last of these possibilities does
not lead to radical expressions.  The other four algebraic cases
are numbered (ii)[a], (ii)[b], (iii), (iv) here.

The reducible case is numbered (i) here.  It is the case when
$\nu=-\mu+\nobreak N$, $N=0,1,2,\dotsc$, and is also called the
degenerate or Gegenbauer case.  It follows from (\ref{eq:Cdef})
and~(\ref{eq:Pdef}) that when $\mu\neq\frac12,1,\frac32,\dotsc$,
\begin{equation}
\label{eq:Gcase}
P_{-\mu+N}^\mu(z) = \frac{2^\mu}{\Gamma(1-\mu)}\,\frac{N!}{(1-2\mu)_N}
\,
(z^2-1)^{-\mu/2}\,
C_N^{1/2-\mu}(z),
\end{equation}
with the same holding if $P_{-\mu+N}^\mu$ and $z^2\nobreak-1$ are replaced
by ${\rm P}_{-\mu+N}^\mu$ and $1-\nobreak z^2$.

Of the four non-icosahedral algebraic cases, the simplest is (ii[a]):
the cyclic case, when the degree~$\nu$ is an integer.  The basic
formulas are
\begin{subequations}
\label{eq:allisaccomplished}
\begin{align}
  P_0^\mu(\coth\xi) &= \Gamma(1-\mu)^{-1}\,e^{\mu\xi},\\
  {\rm P}_0^\mu(\tanh\xi) &= \Gamma(1-\mu)^{-1}\,e^{\mu\xi}.
\end{align}
\end{subequations}
Actually, $P_0^\mu(z),\allowbreak{\rm P}_0^\mu(z)$ are algebraic in~$z$
only if $\mu$~is rational; for general~$\mu$, the term `quasi-cyclic'
will be used.  For any nonzero $\nu\in\mathbb{Z}$, $P_\nu^\mu,{\rm
  P}_\nu^\mu$ are computed from $P_0^\mu,{\rm P}_0^\mu$ by applying ladder
operators that shift the degree.

There is also (ii[b]): the dihedral case, when the order~$\mu$ is a
half-odd-integer.  The basic formulas are
\begin{subequations}
  \begin{align}
    P_\nu^{1/2}(\cosh\xi) &= \sqrt{\frac2\pi}\, \frac{\cosh\left[(\nu+1/2)\xi\right]}{\sqrt{\sinh\xi}},\\
    {\rm P}_\nu^{1/2}(\cos\theta) &= \sqrt{\frac2\pi}\, \frac{\cos\left[(\nu+1/2)\xi\right]}{\sqrt{\sin\theta}},
  \end{align}
\end{subequations}
which define algebraic functions $P_\nu^{1/2},{\rm P}_\nu^{1/2}$ only if
$\nu$~is rational; for general~$\nu$, the term `quasi-dihedral' will be
used.  For any half-odd-integer~$\mu$ other than~$1/2$, $P_\nu^\mu,{\rm
  P}_\nu^\mu$ are computed from $P_\nu^{1/2},{\rm P}_\nu^{1/2}$ by applying
ladder operators that shift the order.

The algebraic cases recently examined\cite{Maier24} include (iii): the
octahedral case, when $(\nu,\mu)\in\mathbb{Z}^2+(\pm1/6,\pm1/4)$.  For
this, define algebraic functions $h_\pm,k_\pm$ trigonometrically by
\begin{subequations}
  \begin{align}
    h_\pm(\cosh\xi) &= \left\{
    (\sinh\xi)^{-1}
\left[
\pm\cosh(\xi/3) + \sqrt{\frac{\sinh\xi}{3\sinh(\xi/3)}}\,
\right]
\right\}^{1/4},
\\
    k_\pm(\cos\theta) &= \left\{
    (\sin\theta)^{-1}
\left[
\cos(\theta/3) \pm \sqrt{\frac{\sin\theta}{3\sin(\theta/3)}}\,
\right]
\right\}^{1/4}.
  \end{align}
\end{subequations}
Then, the basic formulas are
\begin{subequations}
  \label{subeqs:basicocto}
  \begin{align}
  \label{eq:basicocto1}
    P_{-1/6}^{\pm1/4}(\cosh\xi) &=
    3^{(3/8)(1\mp1)}\,\Gamma(1\mp1/4)^{-1}\,h_\pm(\cosh\xi),
    \\
    \label{eq:basicocto2}
    {\rm P}_{-1/6}^{\pm1/4}(\cos\theta) &=
    3^{(3/8)(1\mp1)}\,\Gamma(1\mp1/4)^{-1}\,k_\pm(\cos\theta).
  \end{align}
\end{subequations}
(The plus formulas were derived in Ref.~\citenum{Maier24} and the minus
formulas follow from them, the normalization factors coming from the
condition~(\ref{eq:condition}).)  Ladder operators can be applied to these
basic formulas, as needed.

The other algebraic case recently examined\cite{Maier24} is (iv[a]):
the first subcase of the tetrahedral case, when
$(\nu,\mu)\in\mathbb{Z}^2+(\pm1/4,\pm1/3)$.  For this, define
algebraic functions $f_\pm,g_\pm$ trigonometrically by
\begin{subequations}
  \begin{align}
    f_\pm(\coth\xi) &= \left\{
    (\sinh\xi)
\left[
\pm\cosh(\xi/3) + \sqrt{\frac{\sinh\xi}{3\sinh(\xi/3)}}\,
\right]
\right\}^{1/4},
\label{eq:lastminutea}
\\
    g_\pm(\tanh\xi) &= \left\{
    (\cosh\xi)
\left[
\pm\sinh(\xi/3) + \sqrt{\frac{\cosh\xi}{3\cosh(\xi/3)}}\,
\right]
\right\}^{1/4}.
\label{eq:lastminuteb}
  \end{align}
\end{subequations}
Then, the basic formulas are
\begin{subequations}
  \label{subeqs:basictetra}
  \begin{multline}
    \label{eq:basictetra1}
    P_{-1/4}^{\pm1/3}(\coth\xi) = 2^{1/2\mp3/4}\,3^{-3/8}\,\Gamma(1\mp1/3)^{-1}\\
    {}\times
    \left[
      \sqrt{\sqrt{3}\pm1}\:f_+ \pm \sqrt{\sqrt{3}\mp1}\:f_-
      \right](\coth\xi),
  \end{multline}
  \begin{multline}
    \label{eq:basictetra2}
    {\rm P}_{-1/4}^{\pm1/3}(\tanh\xi) = 2^{1/2\mp3/4}\,3^{-3/8}\,\Gamma(1\mp1/3)^{-1}\\
    {}\times
    \left[\pm
      \sqrt{\sqrt{3}\pm1}\:g_+ + \sqrt{\sqrt{3}\mp1}\:g_-
      \right](\tanh\xi).
  \end{multline}
\end{subequations}
(It was shown in Ref.~\citenum{Maier24} that $Q_{-1/4}^{-1/3}$ is a
multiple of~$f_-$, and $f_+$~is an independent solution of the same
associated Legendre ODE; so $P_{-1/4}^{\pm1/3}$ must be linear combinations
of~$f_+,f_-$, and the coefficients shown in~(\ref{eq:basictetra1}) can be
deduced with some effort from the condition~(\ref{eq:condition}).)  Ladder
operators can be applied to these basic formulas, as needed.

There remains (iv[b]): the second subcase of the tetrahedral case,
when $(\nu,\mu)\in\allowbreak\mathbb{Z}^2+\nobreak(\pm1/6,\pm1/3)$.
This subcase is related to the first tetrahedral one by a quadratic
hypergeometric transformation, but the resulting formulas are
complicated and are not given here.

%\bibliographystyle{ws-rv-van}
%\bibliography{general}

\end{document}